\documentclass[12pt]{article}
\usepackage{a4}
\usepackage{amsthm, amsmath, amsfonts, cite, enumitem}
\usepackage[colorlinks=true,citecolor=black,linkcolor=black]{hyperref}
\usepackage{IEEEtrantools}
\usepackage{epsfig}
\usepackage[intoc]{nomencl}
\makenomenclature

\makeindex

\newtheorem{theorem}{Theorem}
\newtheorem{proposition}[theorem]{Proposition}
\newtheorem{lemma}[theorem]{Lemma}
\newtheorem{corollary}[theorem]{Corollary}

\newtheorem{claim}{Claim}[theorem]

\newcommand{\logconst}{2}
\newcommand{\const}{8}
\newcommand{\bigf}{\omega+2\lceil\logconst{}\log_2(\omega)\rceil+\const{}}

\newcommand{\smallf}{\omega+\lceil\logconst{}\log_2(\omega)\rceil+\const{}}

\newcommand{\I}{\mathcal{I}}

\title{Circle graphs are quadratically $\chi$-bounded}
\author{James Davies\thanks{Department of Combinatorics and Optimization, University of Waterloo, Waterloo, Canada. E-mail: \texttt{jgdavies@uwaterloo.ca}.} \ and Rose McCarty\thanks{Department of Combinatorics and Optimization, University of Waterloo, Waterloo, Canada. E-mail: \texttt{rose.mccarty@uwaterloo.ca}.}}
\date{}
	
\begin{document}

\maketitle
	
\begin{abstract}
	We prove that the chromatic number of a circle graph with clique number $\omega$ is at most $7\omega^2$.
\end{abstract}

\section{Introduction}
We prove the following result.

\begin{theorem}
\label{main-chi}
The chromatic number of a circle graph with clique number $\omega$ is at most $7\omega^2$.
\end{theorem}

\noindent We prove Theorem~\ref{main-chi} as a consequence of the following.

\begin{theorem}
\label{main-perm}
The vertex set of a circle graph with clique number $\omega$ can be partitioned into $7\omega$ sets, each of which induces a permutation graph.
\end{theorem}

\noindent Theorem~\ref{main-perm} implies Theorem \ref{main-chi} since permutation graphs are perfect (see~\cite{golumbicBook}). 

A class of graphs is \textit{$\chi$-bounded} if the chromatic number of each graph in the class is bounded above by some fixed function of its clique number. Such a function is called a \textit{$\chi$-bounding function} for the class. If a polynomial $\chi$-bounding function exists, then the class is \textit{polynomially $\chi$-bounded}. (See Scott and Seymour~\cite{scott2018chiSurvey} for a survey of $\chi$-boundedness.) 

Gy\'arf\'as~\cite{gyarfas1985chromatic} proved that the class of circle graphs is $\chi$-bounded. Kostochka and Kratochv\'il~\cite{kostochka1997covering} showed that there exists a $\chi$-bounding function of order $2^\omega$ for the class of polygon circle graphs, which includes all circle graphs. (A \textit{polygon circle graph} is the intersection graph of a finite set of polygons inscribed in the unit circle.) We further improve on these results by showing that the class of circle graphs is polynomially $\chi$-bounded. Using a result of Krawczyk and Walczak~\cite{krawczyk2017line}, we obtain the following corollary.

\begin{corollary}
The chromatic number of a polygon circle graph with clique number $\omega$ is at most $7\omega^4$.
\end{corollary}

\noindent In fact, the result of Krawczyk and Walczak applies to the more general class of interval filament graphs. An \textit{interval filament graph} is the intersection graph of a finite set of continuous non-negative functions which are defined on a closed interval and map the ends of the interval to zero.

Kostochka~\cite{kostochka1988upper, kostochka2004coloring} showed that every $\chi$-bounding function for the class of circle graphs is $\Omega(\omega \log \omega)$. It would be interesting to close the gap and determine the precise asymptotic growth rate of the optimal $\chi$-bounding function. Tight bounds are known for triangle-free circle graphs; the maximum chromatic number of a triangle-free circle graph is five, with the upper bound due to Kostochka~\cite{kostochka1988upper} and the lower bound due to Ageev~\cite{ageev1996triangle}. 

Our proof approach is partially motivated by a Tur\'{a}n-type theorem for circle graphs of Capoyleas and Pach~\cite{capoyleas1992turan} and the proof of Kim, Kwon, Oum, and Sivaraman~\cite{kim2018classes} that the closure of a hereditary, polynomially $\chi$-bounded class under $1$-joins is also polynomially $\chi$-bounded. Despite this motivation, our proof is self-contained other than the fact that permutation graphs are perfect.

\section{Preliminaries}
For $a,b \in \mathbb{R}$ with $a<b$, we write $(a,b)$ for the open interval on the real line. An \textit{interval system} is a finite set of open subintervals of $(0,1)$ so that no two distinct intervals share an end and no interval has $0$ or $1$ as an end. The \textit{overlap graph $G(\I)$ of an interval system $\I$} is the graph with vertex set $\I$ so that two vertices are adjacent if they overlap as intervals; two intervals \textit{overlap} if they intersect and neither is contained in the other. A graph is a \textit{circle graph} if it is isomorphic to the overlap graph of an interval system. (Equivalently, circle graphs are the intersection graphs of chords on a circle.)

A \textit{pillar} of an interval system $\I$ is a point in $(0,1)$ that is not the end of any interval in $\I$. A graph is a \textit{permutation graph} if it is isomorphic to the overlap graph of an interval system so that there exists a pillar that is contained in every interval. It will be important later that the disjoint union of permutation graphs is also a permutation graph.

Let $\I$ be an interval system. We would like to colour $\I$ with few colours so that each colour induces a permutation graph in $G(\I)$. For any pillar $p$ of $\I$, the set of all intervals in $\I$ that contain $p$ induces a permutation graph. We will choose a finite set $P$ of pillars of $\I$ and assign each interval of $\I$ to one pillar in $P$ that is contained in it. An interval may contain several pillars in $P$, so, to disambiguate the process, we take a total ordering $(P,\prec)$ of $P$ and assign an interval to the first pillar in $P$ under the
given ordering. (Note that $P$ is a set of real numbers, but our ordering $(P,\prec)$ will typically not be the linear ordering given by the reals.)

For this assignment process to work, each interval in $\I$ must contain a pillar in $P$. However, we may require arbitrarily many pillars in order to hit each interval.  Fortunately, the disjoint union of permutation graphs is a permutation graph. So we overcome the problem by colouring $P$ and then colouring each interval with the colour of the pillar it is assigned to. In doing this, one needs to be careful that the monochromatic components remain permutation graphs. To this end, we insist that no two overlapping intervals in $\I$ are assigned to distinct pillars of the same colour.

Let $P$ be a finite set of pillars of $\I$ and let $(P, \prec)$ be a total ordering of $P$. We say an interval $I \in \I$ \textit{is assigned to a pillar $p \in P$ with respect to $(P,\prec)$} if $I$ contains $p$ and all other pillars in $P \cap I$ occur later under $(P,\prec)$. If the ordering is clear, we will just say that $I$ \textit{is assigned to} $p$. Notice that each interval in $\I$ that contains a pillar in $P$ is assigned to exactly one pillar.  If additionally $c:P \rightarrow \mathbb{Z}^+$ is a colouring of $P$, we write $\phi_{(P, \prec, c)}$ for the partial, improper colouring of $\I$ where an interval in $\I$ that is assigned to a pillar $p$ is given the colour $c(p)$.

A \textit{pillar assignment} of an interval system $\I$ is a tuple $(P, \prec, c)$ so that $P$ is a finite set of pillars, $(P, \prec)$ is a total ordering, and $c:P \rightarrow \mathbb{Z}^+$ is a colouring satisfying the following:
\begin{itemize}
    \item[(1)] if $I_1, I_2 \in \I$ overlap and are assigned to pillars $p_1, p_2 \in P$ with $c(p_1)=c(p_2)$, then $p_1 = p_2$.
\end{itemize}
\noindent A pillar assignment $(P, \prec, c)$ is \emph{complete} if every interval in $\I$ contains a pillar in $P$. So for a complete pillar assignment $(P, \prec, c)$, the function $\phi_{(P, \prec, c)}$ colours all intervals in $\I$. Here is the first main observation.

\begin{lemma}
\label{lemma:isPerm}
    For every interval system $\I$, pillar assignment $(P, \prec, c)$ of $\I$, and $k \in \mathbb{N}$, the graph $G(\{I \in \I: \phi_{(P, \prec, c)}(I)=k\})$ is a permutation graph.
\end{lemma}
\begin{proof}
    Let $\I' \subseteq \I$ so that $G(\I')$ is a connected component of $G(\{I \in \I: \phi_{(P, \prec, c)}(I)=k\})$. By the definition of a pillar assignment, there exists $p \in P$ so that every interval in $\I'$ contains $p$. Thus $G(\I')$ is a permutation graph. The lemma follows since the disjoint union of permutation graphs is a permutation graph.
\end{proof}

In the next section, we will prove that every interval system whose overlap graph has clique number $\omega$ has a complete pillar assignment using at most $\bigf{}$ colours. By Lemma~\ref{lemma:isPerm} and since $\bigf{} \leq 7\omega$ for $\omega \geq 2$, this implies Theorem~\ref{main-perm}. (Note that Theorem~\ref{main-perm} is trivially true for $\omega = 1$). Next we will prove a lemma that roughly says that not too many pillars can affect each other, but first we need some more definitions.

For a finite set $P \subset (0,1)$, a \textit{segment of $P$} is an open interval with ends in $P \cup \{0,1\}$ that contains no point in $P$. So there is a unique partition of $(0,1)\setminus P$ into $|P|+1$ segments.

Now, let $\I$ be an interval system and let $P$ be a finite set of pillars of $\I$. Notice that each end of an interval in $\I$ is contained in a segment of $P$. When the interval system is clear, for pillars $p_1<p_2$, we say that the \textit{$P$-degree of $(p_1, p_2)$} is the number of sets $\{S_1, S_2\}$ so that
\begin{itemize}\itemsep0em
    \item[(1)] $S_1$ is a segment of $P$ that is disjoint from $(p_1,p_2)$,
    \item[(2)] $S_2$ is a segment of $P \cup \{p_1,p_2\}$ that is contained in $(p_1,p_2)$, and
    \item[(3)] there exists an interval in $\cal I$ with one end in $S_1$ and one end in $S_2$.
\end{itemize}

\noindent We happen to only apply this definition when either $p_1,p_2 \in P$ or $(p_1, p_2)$ is contained in a segment of $P$. We denote the $P$-degree of $(p_1, p_2)$ by $d_P(p_1,p_2)$. If $J=(p_1, p_2)$ we will write $d_P(J)$ for $d_P(p_1, p_2)$. Our next lemma can be seen as a permutation graph version of a theorem of Capoyleas and Pach~\cite{capoyleas1992turan}. 

\begin{lemma}
\label{lemma:counting}
    Let $\I$ be an interval system whose overlap graph has clique number $\omega$. Let $P$ be a finite set of pillars of $\I$, and let $p_1, p_2 \in P$ with $p_1<p_2$. Then $d_P(p_1, p_2)\leq \omega|P|$.
\end{lemma}

\begin{proof}
    We proceed by induction on $|\I|+|P|$. If $|\I|=0$, then $d_P(p_1, p_2)=0$, and the lemma holds. By induction, we may assume that all intervals in $\I$ overlap with $(p_1, p_2)$, as we can delete all other intervals from $\I$ without changing $d_P(p_1, p_2)$. 
    
    \begin{claim}
    The graph $G(\I)$ is a permutation graph.
    \end{claim}
    \begin{proof}
    For each $i\in \{1,2\}$, define $\I_i$ as the set of all intervals in $\I$ that contain $p_i$. Then $\I$ is the disjoint union of $\I_1$ and $\I_2$. Let $a$ be the largest end of an interval in $\I_1$, subject to satisfying $a<p_1$. Let $b$ be the largest end of an interval in $\I_2$. Now fix some $\epsilon>0$ so that $1-\epsilon a >b$. Then $\{(d, 1-\epsilon c): (c,d) \in \I_1\}\cup \I_2$ is an interval system with overlap graph isomorphic $G(\I)$ so that every interval contains $p_2$.
    \end{proof}
    
    By the claim, $G(\I)$ is perfect and so has a proper colouring using $\omega$ colours. If $\omega\geq 2$, then the lemma holds by induction, applying the lemma separately to the intervals of each colour. So we may assume that $\omega =1$. 
    
    If $|P|=2$, then $d_P(p_1, p_2)\leq 2 = \omega |P|$, and the lemma holds. Then by induction, we may assume that each segment of $P$ contributes at least two to $d_P(p_1,p_2)$. Let $S_1$ and $S_2$ be the two segments of $P$ having $p_1$ as an end. For each $i\in \{1,2\}$, there is an interval $I_i \in \I$ that has an end in $S_i$ and no end in $S_{3-i}$. Then $I_1$ and $I_2$ overlap, a contradiction to the fact that $\omega = 1$.
\end{proof}

Lemma~\ref{lemma:counting} is the only way we will use the clique number of the overlap graph. Let $\I$ be an interval system whose overlap graph has clique number $\omega$. We will choose a pillar assignment $(P, \prec,c)$ using at most $\bigf{}$ colours so that $\phi_{(P, \prec, c)}$ colours as many intervals in $\I$ as possible, subject to satisfying an additional property. The additional property will be similar to saying that every segment of $P$ has $P$-degree at most $\smallf{}$. However we need a different notion which incorporates the ordering of the pillars. We give this definition next.

Let $\I$ be an interval system, let $P$ be a finite set of pillars of $\I$, and let $(P, \prec)$ be a total ordering of $P$. For an open subinterval $J$ of $(0,1)$ that is contained in a segment of $P$, the \textit{$(P, \prec)$-degree of $J$}, denoted $d_{(P, \prec)}(J)$, is the number of pillars $p \in P$ so that there exists an interval in $\I$ with an end in $J$ that is assigned to $p$ with respect to $(P, \prec)$. If $J = (p_1, p_2)$ we will write $d_{(P, \prec)}(p_1,p_2)$ for $d_{(P,\prec)}(J)$. Observe that $d_{(P,\prec)}(J)\leq d_P(J)$ since two intervals in $\I$ with ends in the same two segments of $P$ are assigned to the same pillar with respect to $(P,\prec)$. We say that the \textit{maximum degree of $(P,\prec)$} is the maximum $(P,\prec)$-degree of a segment of $P$. We now prove one final lemma.

    \begin{lemma}
    \label{lemma:unordered}
        Let $\I$ be an interval system, let $k \in \mathbb{Z}^+$, and let $P^*$ be a set of at most $2^k-1$ pillars of $\I$. Then there exist a total ordering $(P^*, \prec^*)$ of $P^*$ and a colouring $c^*:P^* \rightarrow \{1,2,\ldots, k\}$ so that $(P^*, \prec^*, c^*)$ is a pillar assignment and the maximum degree of $(P^*, \prec^*)$ is at most $k$.
    \end{lemma}
    
    \begin{proof}
    	This is trivially true for $k=1$. We will argue by induction on $k$. Choose $p\in P^*$ such that the intervals $(0,p)$ and $(p,1)$ each contain at most $2^{k-1}-1$ pillars in $P^*$. Let $P_1 \subset P^*$ be the set of pillars less than $p$ and $P_2\subset P^*$ the set of pillars greater than $p$.
    	
    	By induction, for each $i\in \{1,2\}$, there is a total ordering $(P_i, \prec_i)$ of $P_i$ and colouring $c_i:P_i \rightarrow \{1,2,\ldots, k-1\}$ so that $(P_i, \prec_i, c_i)$ is a pillar assignment and the maximum degree of $(P_i, \prec_i)$ is at most $k-1$. Let $\prec^*$ be the total ordering of $P^*$ obtained from $\prec_1 \cup \prec_2$ by adding $p$ first. Let $c^*$ be the colouring of $P^*$ that extends $c_1$ and $c_2$ and has $c^*(p)=k$. Then any two pillars $p_1 \in P_1$ and $p_2 \in P_2$ are separated by the preceding pillar $p$, which is the only pillar of colour $k$. It follows that $(P^*, \prec^*, c^*)$ is a pillar assignment and the maximum degree of $(P^*, \prec^*)$ is most $k$.
    \end{proof}

\section{Main result}
In this section we prove the following, which implies Theorem~\ref{main-perm}.

\begin{proposition}
\label{thm:extends}
Every interval system $\I$ whose overlap graph has clique number $\omega$ has a complete pillar assignment using at most $\bigf{}$ colours.
\end{proposition}

\begin{proof}
    The proposition is trivially true if $\omega \leq 1$. So fix an interval system $\I$ whose overlap graph has clique number $\omega \geq 2$. Throughout the proof, all colours are in $\{1,2,\ldots, \bigf{} \}$. Choose a pillar assignment $(P, \prec, c)$ using these colours so that the maximum degree of $(P, \prec)$ is at most $\smallf{}$. Subject to this, choose $(P, \prec, c)$ so that $\phi_{(P, \prec, c)}$ colours as many intervals in $\I$ as possible. Such a pillar assignment exists. Suppose, for a contradiction, that $(P, \prec, c)$ is not complete.
    
    Thus, there is an interval $I \in \I$ that is contained in some segment $S=(p^-, p^+)$ of $P$. Let $P_1$ denote the set of all $p \in P$ so that there exists an interval in $\I$ with an end in $S$ that is assigned to $p$. So $|P_1|=d_{(P, \prec)}(S)$. 
    
    \begin{claim}
    There exists a set $P^* \subset S$ of at most $\omega^2-1$ pillars of $\I$ so that $I$ contains a pillar in $P^*$ and the $P_1$-degree of each segment of $P \cup P^*$ that is contained in $S$ is at most $\omega+\const{}$.
    \end{claim}
    \begin{proof}
        Choose pillars $p^- = p_0^* < p_1^* <\ldots<p_{t+1}^*=p^+$ so that $d_{P_1}(p_0^*, p_1^*)=\ldots = d_{P_1}(p_{t-1}^*, p_{t}^*)=\omega+\const{}$ and $1\leq d_{P_1}(p_t^*, p_{t+1}^*)\leq \omega+8$. Let $P_1^* = \{p_i^*: i \in \{1,2,\ldots, t\}\}$. Let $p_1$ be the largest pillar in $P_1$ that is at most $p^-$, as ordered by the real line. Similarly, let $p_2$ be the smallest pillar in $P_1$ that is at least $p^+$, as ordered by the real line. Then by Lemma~\ref{lemma:counting},
        \begin{align*}
        |P_1^*|(\omega+\const{}) &< d_{P_1 \cup P_1^*}(p_1, p_2)
           \\&\leq \omega(|P_1|+|P_1^*|)
           \\&\leq \omega(\smallf{}+|P_1^*|).
        \end{align*} It follows then, since $\omega \geq 2$, that $|P_1^*|<\omega^2-1$. Then the claim holds with $P^* = P_1^* \cup \{p^*\}$, where $p^*$ is any pillar contained in $I$.
    \end{proof}
    Let $P^*$ be a set of pillars as in the claim. By Lemma~\ref{lemma:unordered}, there exist a total ordering $(P^*, \prec^*)$ of $P^*$ and a colouring $c^*$ of $P^*$ using at most $\lceil \log_2(\omega^2)\rceil = \lceil 2\log_2(\omega)\rceil$ colours so that $(P^*, \prec^*, c^*)$ is a pillar assignment and the maximum degree of $(P^*, \prec^*)$ is at most $\lceil 2\log_2(\omega)\rceil$. Choose the colours so that no pillar in $P_1$ is the same colour as a pillar in $P^*$. Then $(P \cup P^*, \prec\cup \prec^*, c \cup c^*)$ is a pillar assignment using at most $\bigf{}$ colours.
	
	We claim that the maximum degree of $(P \cup P^*, \prec\cup \prec^*)$ is at most $\smallf{}$. Let $S^*$ be a segment of $P \cup P^*$. If $S^*$ is disjoint from $S$, then no interval in $\I$ with one end in $S^*$ and one end in $S$ is assigned to a pillar in $P^*$ with respect to $(P \cup P^*, \prec\cup \prec^*)$. So in this case the $(P \cup P^*, \prec\cup \prec^*)$-degree of $S^*$ equals the $(P, \prec)$-degree of $S^*$, which is at most $\smallf{}$. 
	
	Now suppose that $S^*$ is contained in $S$. Let $(P_1,\prec_1)$ denote the restriction of $(P, \prec)$ to $P_1$. Observe that $d_{(P, \prec)}(S^*)\leq d_{(P_1, \prec_1)}(S^*)$ by the definition of $P_1$. Then the $(P \cup P^*, \prec\cup \prec^*)$-degree of $S^*$ is at most \begin{align*}d_{(P, \prec)}(S^*)+d_{(P^*, \prec^*)}(S^*)&\leq d_{(P_1, \prec_1)}(S^*)+d_{(P^*, \prec^*)}(S^*)\\&\leq d_{P_1}(S^*)+d_{(P^*, \prec^*)}(S^*) \\&\leq \smallf{}.\end{align*}This shows a contradiction to the choice of $(P, \prec, c)$ and completes the proof.
\end{proof}

We remark that the constants on the smaller order terms in Proposition~\ref{thm:extends} can be slightly improved for $\omega$ sufficiently large. However it seems that new ideas would be needed to find a pillar assignment using $(1-\epsilon)\omega$ colours for a fixed $\epsilon>0$ and $\omega$ sufficiently large. It would be interesting to know, at least asymptotically, the smallest number of sets needed to partition the vertex set of any circle graph with clique number $\omega$ into permutation graphs.

\section*{Acknowledgments}
The authors would like to thank Jim Geelen for valuable discussions on circle graphs and helpful comments on the presentation.

\bibliography{Circlegraphbib}

\end{document}